\documentclass[10pt]{article}

\usepackage{amsmath, amsthm, amsfonts, amssymb}
\usepackage[notref,notcite]{showkeys}

\newtheorem{theorem}{Theorem}[section]
\newtheorem{corollary}[theorem]{Corollary}
\newtheorem{lemma}[theorem]{Lemma}
\newtheorem{proposition}[theorem]{Proposition}
\newtheorem{example}[theorem]{Example}
\newtheorem{remarks}[theorem]{Remarks}
\newtheorem{remark}[theorem]{Remark}
\newtheorem{definition}[theorem]{Definition}

\numberwithin{equation}{section}

\title{\bf  Remarks on Li-Yau inequality on graphs}
\author{Bin Qian\thanks{Department of Mathematics, Changshu Institute of Technology, Changshu, Jiangsu 215500,  China.
E-mail: binqiancs@yahoo.com, binqiancn@gmail.com.
 Partially supported by the NSF of China (Grant NO. 11201040)
 }}\date{
 }

\begin{document}
\maketitle

\begin{abstract}
In this paper, we study  Li-Yau gradient estimates for the solutions $u$ to the heat equation $\partial_tu=\Delta u$ on graphs under the curvature condition $CD(n,-K)$ introduced by Bauer et al. in \cite{BHLLMY}. As applications, we derive Harnack inequalities and heat kernel estimates on graphs. Also we present a type of Hamilton gradient estimates.

\end{abstract}


\noindent \hskip 24pt {\bf AMS Classification Subjects 2010:} 05C81 53C21 35K08

\section{Introduction}
In their celebrated work, Li and Yau \cite{LY} proved an upper bound on the gradient of positive solutions to the heat equation, called Li-Yau inequality.  This inequality is a very powerful tool to study estimation of heat kernels. More precisely, in its simplest form, it asserts that, for an n-dimensional compact Riemannian manifold with non-negative curvature, if $u$ is a positive solution to the heat equation $\partial_tu=\Delta u$, then
\begin{equation*}\label{LY}
\frac{|\nabla u|^2}{u^2}-\frac{u_t}{u}\le \frac{n}{2t}.
\end{equation*}

Many generalizations of this inequality have been developed, see \cite{Da, BQ99,LX11,Li1,Qian,BBBQ} and references therein.

Recently, Bauer et al \cite{BHLLMY} prove a discrete version of Li-Yau inequality on graphs via introducing a new notion of curvature, a type of  chain rule formula  for graph and a discrete version of maximum principle. More precisely,

{\bf Theorem }(Due to \cite{BHLLMY}) Let $G=(V,E)$ be a finite graph satisfying $CDE(n,-K)$ with $K\ge0$, and let $u $ be a positive solution to the heat equation on $G$. Then  for fixed $0<\alpha<1$, we have  for all $t>0$,
$$
\frac{(1-\alpha)\Gamma(\sqrt{u})}{u}-\frac{\frac{\partial}{\partial t}(\sqrt{u})}{\sqrt{u}}\le \frac{n}{2(1-\alpha)t}+\frac{Kn}{\alpha}.
$$
In particular, if $K=0$, we can take $\alpha=0$.

Meanwhile, for the Laplacian on  the manifolds with negative curvature, the parameter $\alpha$ can be replaced by some function of the time $t$, for example,

{\bf Theorem} (Due to \cite{BQ99,LX11}) Let $M$ be a complete Riemannian manifold with dimension $n$. Assume $Ricci(M)\ge -K$ with $K\ge0$. For any solution $u$ to the heat equation $\partial_tu=\Delta u$, we have for $t>0$,
$$
|\nabla \log u|^2-\left(1+\frac{2}{3}Kt\right)(\log u)_t\le \frac{n}{2t}+\frac{nK}{2}\left(1+\frac{1}{3}Kt\right),
$$
and 
$$
|\nabla \log u|^2-\left(1+\frac{\sinh(Kt)\cosh(Kt)-Kt}{\sinh^2(Kt)}\right)(\log u)_t\le \frac{nK}{2}\left(1+\coth(Kt)\right).
$$
As a consequence, new bounds with explicit constants for the associated heat kernels can be derived, see \cite{Ham,BQ99,LX11,Qian}. So it is nature to ask that in the setting of graph, whether  a similar result as above holds in the setting of graphs. This is the starting point of this paper.

Let $G=(V,E)$ be a graph. Given a finite measure $\mu:V\to \mathbb{R}$ on $V$, the $\mu$-Laplacian on  $G$ is the operator $\Delta: \mathbb{R}^{|V|}\to \mathbb{R}^{|V|}$ defined by $$
\Delta f(x)=\frac{1}{\mu(x)}\sum_{y\sim x}\omega_{x y}(f(y)-f(x)).
$$

To introduce the new notion of curvature in \cite{BHLLMY},  we recall the $\Gamma$ (or $\mu$-$\Gamma$) gradient operator, which is defined as, see \cite{Ba97} etc,
\begin{align*}
2\Gamma(f,g)&=2\langle\nabla f, \nabla g\rangle (x)=\left(\Delta(fg)-f\Delta g-g\Delta f \right)(x)\\
&=\frac{1}{\mu(x)}\sum_{y\sim x}\omega_{xy}\left(f(y)-f(x)\right)\left(g(y)-g(x)\right).
\end{align*}
We denote $\Gamma(f)=\Gamma(f,f)=|\nabla f|^2$ for short. For any positive $f$, we have
\begin{equation}\label{id12}
\Delta (\sqrt{f})=\frac{\Delta f}{2\sqrt{f}}-\frac{\Gamma({\sqrt{f}})}{\sqrt{f}},
\end{equation}
this identity will play an important role in the following, refer to \cite{BHLLMY} in detail. Now let us recall the new curvature on graph introduced in \cite{BHLLMY}:
\begin{definition}\label{cde}
We say that a graph $G$ satisfies the exponential curvature dimension inequality $CDE(n,K)$ if for vertex $x\in V$ and any positive function $f:V\to \mathbb{R}$ such that $\Delta f(x)<0$ we have
$$
\Gamma_2(f)-\Gamma\left(f,\frac{\Gamma(f)}{f}\right)\ge \frac1n(\Delta f)^2+K\Gamma(f),
$$
or equivalently
$$
\widetilde{\Gamma}_2(f):=\frac12\Delta\Gamma(f)-\Gamma\left(f,\frac{\Delta(f^2)}{2f}\right)\ge\frac1n(\Delta f)^2+K\Gamma(f).
$$
\end{definition}
In the case of diffusion operators on the Riemanian manifold, the curvature condition dimension condition $CD(K,n)$ (refer to \cite{Ba97} for its definition) implies $CDE(n,K)$, see \cite{BHLLMY}.

Consider the following heat equation on graph $G=(V,E)$
\begin{equation}\label{heat}
\partial_tu(x,t)=\Delta u(x,t), (x,t)\in V\times (0,\infty)
\end{equation}
with the initial data $u(\cdot,0)=u_0$. The solution $u$ can be written as $u(x,t)=P_tu_0$ where $P_t=e^{t\Delta}$ is the heat semigroup and $u_0=u(\cdot,0)$. We shall study the Li-Yau gradient estimates  and Hamilton gradient estimates of any positive solution $u$, see section \ref{sec-LY} and section \ref{sec-Ham}. In section \ref{sec-mp}, we present some maximum principles on graph, which will be used to get the global gradient estimates. We state the Harnack inequalities and heat kernel estimates in section \ref{sec-Harn}.

{\bf Notations:} For $\mu-$Laplacian $\Delta$, assume $\omega_{min}:=\inf_{\omega\in E}\omega_{e}>0$, $D_{\omega}:=\max_{\stackrel{x,y\in V}{x\sim y}}\frac{\deg(x)}{\omega_{xy}}<\infty$, $D_{\mu}:=\max_{x\in V}\frac{\deg(x)}{\mu(x)}<\infty$. For any positive solution $u(x,t)$ to the heat equation \eqref{heat} and $T>0$, denote by $S_T$  the set $\{(x,t)\in V\times (0,T]|\Delta\sqrt{u}(x,t)<0\}$.

\section{Auxiliary Propositions}\label{sec-mp}
In this section, we give some auxiliary propositions of their only interests, which can be seen as conditional maximum principles in discrete settings. They will be used  heavily in the following sections. The proofs is similar to the one in the  setting of manifolds, see for example \cite{Gri}.

\begin{lemma}[Strong maximum primciple]\label{mp1}
Let $G=(V,E)$ be a finite graph, $u$ is a positive solution to the heat equation \eqref{heat}, the set $S_T$ is defined as above. For some function $F$  satisfying
\begin{equation}\label{diff-0}(\Delta-\partial_t)F(x,t)>0, \ \ \forall\  (x,t)\in S_T, \end{equation}
with $F(x,0)\le 0, \ \forall\ x\in V$, assume $F(x,t)\le0$ holds for $(x,t)\in S_T^c$ and $t>0$. Then for any $(x,t)\in V\times [0,T]$, we have  $F(x,t)\le0$.
\end{lemma}
\begin{proof}
For some $t^*>0$, assume $(x^*,t^*)$ is the place where $F$ attains its maximum in the $V\times [0,T]$ domain. We may assume $(x^*,t^*)\in S_T$ otherwise there is nothing to prove. It follows
$\partial_tF(x^*,t^*)=0$ if $t^*\in (0,T)$ and $\partial_tF(x)\ge0$ if $t^*=T$. Meanwhile
$$
\  \Delta u(x^*,t^*)=\frac{1}{\mu(x)}\sum_{y\sim x^*}w_{x^*y}(u(y,t^*)-u(x^*,t^*))\le 0,
$$
 thus $(\Delta-\partial_t)u(x^*,t^*)\le 0$. The assumption \eqref{diff-0} yields a contradiction.  Hence $t^*=0$, the desired result follows from the assumption that $F(x,0)\le0$ for all $x\in V$.
\end{proof}

Now let us state the weak maximum principle as follows:

\begin{proposition}\label{mp2} The conclusion in the above Lemma \ref{mp1} holds if we replace \eqref{diff-0} by
\begin{equation}\label{diff-1}(\Delta-\partial_t)F(x,t)\ge0, \ \ \forall\  (x,t)\in S_T.\end{equation}

\end{proposition}
\begin{proof} For any $\varepsilon>0$,  denote $F_{\varepsilon}=F(x,t)-\varepsilon t$, we have, for $(x,t)\in S_T$, \begin{equation*}
(\Delta-\partial_t)F_{\varepsilon}=(\Delta-\partial_t)F+\varepsilon\ge \varepsilon>0,
\end{equation*}
and $F_{\varepsilon}(x,t)\le0$ holds for $(x,t)=(x,0)$ with $x\in V$ or $(x,t)\in S_T^c$ and $t>0$. Hence by Lemma \ref{mp1}, we have $F_{\varepsilon}(x,t)\le 0$ for any $(x,t)\in V\times [0,T]$. Letting $\varepsilon\to 0$, we complete the proof.
\end{proof}

Let us extend the above results to the case of infinite graph.
\begin{proposition}\label{mp3}
Let $G=(V,E)$ be a infinite graph, $u$ is some positive solution to the heat equation \eqref{heat}, the set $S_T$ is defined as above. For some function $F$  satisfying
\begin{equation}\label{diff-2}(\Delta-\partial_t)F(x,t)\ge0, \ \ \forall\  (x,t)\in S_T, \end{equation}
with $F(x,0)\le 0, \ \forall\ x\in V$, assume $F(x,t)\le0$ holds for $(x,t)\in S_T^c$ and $t>0$. Then for any $(x,t)\in V\times [0,T]$, we have  $F(x,t)\le0$.

\end{proposition}
\begin{proof}
For any $\varepsilon>0$,  denote $F_{\varepsilon}=F(x,t)-\varepsilon t$, we have, for $(x,t)\in S_T$, \begin{equation}\label{diff-3}
(\Delta-\partial_t)F_{\varepsilon}(x,t)=(\Delta-\partial_t)F+\varepsilon\ge \varepsilon,
\end{equation}
and $F_{\varepsilon}(x,t)\le0$ holds for $(x,t)=(x,0)$ with $x\in V$ or $(x,t)\in S_T^c$ and $t>0$. For some $t^*>0$, we can find a sequence $\{x_k^*\}_{k\in \mathbb{N}}$ such that $$F_{\varepsilon}(x_k^*,t^*)\ge F_{\varepsilon}(y,t^*)-\frac{1}{k}, \ \forall y\in V,\ \ \mbox{and} \ \   F_{\varepsilon}(y,t)\le F_{\varepsilon}(y,t^*),\  \forall  y\in V,t\in [0,T].$$

If there is a subsequence $\{x_{n_k}^*\}_{k\in \mathbb{N}}$ of $\{x_k^*\}_{k\in \mathbb{N}}$  such that for any $k\in \mathbb{N}$,  $(x_{n_k}^*,t^*)\in S_T^c$. Then for any $y\in V, t\in [0,T]$, $F_{\varepsilon}(y,t)\le F_{\varepsilon}(y,t^*)\le F_{\varepsilon}(x_{n_k}^*,t^*)+\frac{1}{n_k}\le\frac{1}{n_k}$, letting $k\to\infty$ then ${\varepsilon\to 0}$, the desired result follows. Hence we can assume for all $k\in\mathbb{N}$, $(x_k,t^*)\in S_T$, in this case, we have
$$
\Delta F_{\varepsilon}(x_k^*,t^*)=\frac{1}{\mu(x_k^*)}\sum_{y\sim x_k^*}\omega_{x_k^*y}\left(F_{\varepsilon}(y,t^*)-F_{\varepsilon}(x_k^*,t^*)\right)\le\ \frac{D_{\mu}}{k},\ \partial_tF_{\varepsilon}(x_k^*,t^*)\ge 0.
$$
Thus $(\Delta-\partial_t)F_{\varepsilon}(x_k^*,t^*)\le \frac{D_{\mu}}{k}$, it follows $(\Delta-\partial_t)F_{\varepsilon}(x_k^*,t^*)\le \frac{\varepsilon}{2}<\varepsilon$ for $k$ large enough, which is contradictory to \eqref{diff-3}. So it should be $t^*=0$. It follows $ F_{\varepsilon}(y,t)\le F_{\varepsilon}(y,0)$ holds for all  $y\in V,t\in [0,T]$. Letting $\varepsilon\to0$, we complete the proof.                   \end{proof}

\begin{remarks}\label{rem1}
(a). The set $S_T$ can be replaced by any other set.

(b). The above three Propositions hold if we replace $\Delta$ by $\Delta-q$, where the potential function $q$ is nonnegative.
\end{remarks}


\section{Li-Yau gradient estimates}\label{sec-LY}
To state the main results in this section, let us introducte some notations. For a given $C^1$ positive function
 $a(t):(0,\infty)\to (0,\infty)$, we always suppose $a(t)$  satisfies the following assumptions:\\
 (A1). For all $t>0$, $a(t)>0,a'(t)>0$ and $\lim_{t\to0}a(t)=0$, $\lim\limits_{t\to0}\frac{a(t)}{a'(t)}=0$. \\
 (A2). For any $L>0$,
 $\frac{a'^2}{a}$ is  continuous and integrable on the interval $[0,L]$.

\subsection{Global estimates}

 The main result in this subsection is the following global Li-Yau gradient estimate:
\begin{theorem}\label{LY1} Let $G=(V,E)$ be a finite (or infinite) graph satisfying the curvature dimension condition $CDE(n,-K)$ with $K\ge0$, and let $u$ be a positive solution to the heat equation \eqref{heat} on $G$. Then for all $t>0$,
\begin{equation}\label{LY-1}
\frac{\Gamma(\sqrt{u})}{u}-\alpha(t)\frac{\partial_t (\sqrt{u})}{\sqrt{u}}\le \frac{1}{2}\varphi(t),
\end{equation}
where \begin{equation}\label{notation1}
\alpha(t)=\frac{2K}{a(t)}\int_0^ta(s)ds+1, \
\varphi(t)=nK+\frac{nK^2}{a(t)}\int_0^ta(s)ds+\frac{n}{4a(t)}\int_0^t\frac{a'^2(s)}{a(s)}ds.
 \end{equation}
\end{theorem}

\begin{remarks} (a). In the setting of Laplacian on Riemannian manifolds, a similar result has been obtained in \cite{Qian}, which generalized the work of Li-Yau \cite{LY}, Li-Xu \cite{LX11} and etc. See also \cite{Li1} for diffusion operators on Riemannian manifolds and \cite{BBBQ,Qian1} for subelliptic operators.

(b). For the case of Schr\"odinger operators $\Delta-q$ with nonnegative potential $q$, assume some additional condition on $q$, a similar result of Theorem \ref{LY1} holds.
\end{remarks}

Let us give some examples, which are similar to \cite{Qian}. But the under space in \cite{Qian} is Riemannian manifolds.
\begin{example}
\begin{description}
\item{(1).} Taking $a(t)=t^{\gamma}$ with $\gamma>1$, \eqref{LY-1} reduces to
\begin{equation}\label{ex1}
\frac{\Gamma(\sqrt{u})}{u}-\left(1+\frac{2Kt}{1+\gamma}\right)\frac{\partial_t (\sqrt{u})}{\sqrt{u}}\le \frac{nK}{2}+\frac{nK^2t}{2(1+\gamma)}+\frac{n\gamma^2}{8(\gamma-1)t}.
\end{equation}
In particular, choose $\gamma=2$, we have
\begin{equation}\label{ex1-1}
\frac{\Gamma(\sqrt{u})}{u}-\left(1+\frac{2Kt}{3}\right)\frac{\partial_t (\sqrt{u})}{\sqrt{u}}\le \frac{nK}{2}+\frac{nK^2t}{6}+\frac{n\gamma^2}{8t}.
\end{equation}

Furthermore, if the graph $G$ satisfies the curvature dimension inequality $CDE(n,0)$, we have
$$
\frac{\Gamma(\sqrt{u})}{u}-\frac{\partial_t (\sqrt{u})}{\sqrt{u}}\le\frac{n}{2t}.
$$
It has been observed in \cite{BHLLMY}, Theorem 4.3.
\end{description}

\item{(2).} Taking $a(t)=t^2+\gamma t^3$ with $\gamma\ge0$, \eqref{LY-1} reduces to
$$
\frac{\Gamma(\sqrt{u})}{u}-\left(1+\frac23Kt-\frac{\gamma Kt^2}{6(1+\gamma
t)}\right)\frac{\partial_t(\sqrt{u})}{\sqrt{u}}\le \frac{nK}{2}+\frac{nK^2(4t+3\gamma
t^2)}{24(1+\gamma t)}+\frac{n(9\gamma^2t^2+6\gamma t+2\ln(1+\gamma
t))}{16\gamma(t^2+\gamma t^3)}.
$$
\item{(3).} Taking $a(t)=\sinh^2(Kt)$, \eqref{LY-1} reduces to
$$
\frac{\Gamma(\sqrt{u})}{u}-\left(1+\frac{\sinh(Kt)\cosh(Kt)-Kt}{\sinh^2(Kt)}\right)
\frac{\partial_t(\sqrt{u})}{\sqrt{u}}\le \frac{nK}{2}(\coth(Kt)+1).
$$

 In this case,
$\alpha(t)=1+\frac{\sinh(Kt)\cosh(Kt)-Kt}{\sinh^2(Kt)}$ is bounded for all $t>0$ while the ones in \eqref{ex1} and \eqref{ex1-1} are unbounded. In the setting of Riemannian manifold, this type of estimate is firstly observed by Li and Xu in \cite{LX11}.

\item{(4).} Take $a(t)=(e^{\gamma Kt}-1)^2$ with $\gamma\neq0$, \eqref{LY-1} becomes
$$
\frac{\Gamma(\sqrt{u})}{u}-\left(1+\frac{e^{2\gamma Kt}-4e^{\gamma Kt}+2\gamma
Kt+3}{\gamma(e^{\gamma Kt}-1)^2}\right)\frac{\partial_t(\sqrt{u})}{\sqrt{u}}\le
\frac{nK\left(\left((\gamma+1)e^{\gamma Kt}-2\right)^2+2\gamma
Kt-(\gamma-1)^2\right)}{4\gamma(e^{\gamma Kt}-1)^2}.
$$
For $\gamma>0$, $\alpha(t)=1+\frac{e^{2\gamma Kt}-4e^{\gamma Kt}+2\gamma
Kt+3}{\gamma(e^{\gamma Kt}-1)^2}$ is bounded for all $t>0$, while for
$\gamma<0$, $\alpha(t)$ is also bounded for $t\in [0,T]$
with  fixed $T>0$.
\item{(5).} Taking $ a(t)=e^{-\frac{2Kt}{1+\beta}}\left(1-e^{-\frac{2Kt}{1+\beta}}\right)^{\beta}, \mbox{ with } \beta>1. $ \eqref{LY-1} reduces to
\begin{equation*}
\frac{\Gamma(\sqrt{u})}{u}-e^{\frac{2Kt}{1+\beta}}\frac{\partial_t(\sqrt{u})}{\sqrt{u}}\le
\frac{nK\beta^2}{4(\beta-1)(\beta+1)}\frac{e^{\frac{4Kt}{\beta+1}}}{e^{\frac{2Kt}{\beta+1}}-1}, \ \mbox{ where }\  t<\frac{(1+\beta)\log (1+\beta)}{2K}.
\end{equation*}

If $G$ satisfies the curvature dimension inequality $CDE(n,K)$ with $K\ge0$, we can take  $ a(t)=e^{\frac{2Kt}{1+\beta}}\left(e^{\frac{2Kt}{1+\beta}}-1\right)^{\beta}$,  with  $\beta>1$.  \eqref{LY-1} reduces to
\begin{equation*}
\frac{\Gamma(\sqrt{u})}{u}-e^{-\frac{2Kt}{1+\beta}}\frac{\partial_t(\sqrt{u})}{\sqrt{u}}\le
\frac{nK\beta^2}{4(\beta-1)(\beta+1)}\frac{e^{-\frac{4Kt}{\beta+1}}}{1-e^{-\frac{2Kt}{\beta+1}}},\ \ t>0.
\end{equation*}
\end{example}

As an application of Theorem \ref{LY1}, we have the following Liouville property:
\begin{corollary}\label{lv1}
Suppose the finite (or infinite) graph  $G$ satisfies the curvature dimension inequality $CDE(n,0)$, let $u$ be any positive solution to the equation $\Delta u=0$, then $u$ is a constant.
\end{corollary}
\begin{proof}
By \eqref{LY-1}, it is easy to see that $\Gamma(\sqrt{u})(x)=0$, for all $x\in V$, hence for any $x\in V$, $u(y)=u(x)$ for $y\sim x$. Thus $u$  must be a constant.
\end{proof}
To prove Theorem \ref{LY1}, let us first give the following lemma.

\begin{lemma}\label{lem1} Let $ H=a(t)\left(\frac{2\Gamma(\sqrt{u})-\alpha(t)\Delta u}{u}-\varphi(t)\right),$
where $a(t),\alpha(t),\varphi(t)$ are some smooth enough functions. We have
\begin{align}\label{diff1}
(\Delta-\partial_t)(uH)=4a(t)\widetilde{\Gamma}_2(\sqrt{u})+(a\alpha)'\Delta u-2a'(t)\Gamma(\sqrt{u})+(a \varphi)' u.
\end{align}

\end{lemma}
\begin{proof}
Direct computation gives
\begin{align*}
\partial_t(u H)=a' \left(2\Gamma(\sqrt{u})-\alpha \Delta u-\varphi u \right)+a \left(4\Gamma\left(\sqrt{u},\frac{\Delta u}{2\sqrt{u}}\right)-\alpha \Delta \partial_tu-\alpha' \Delta u-\varphi' u-\varphi \partial_tu\right),
\end{align*}

\begin{align*}
\Delta(uH)&=a(t)\left(2\Delta\Gamma(\sqrt{u})-\alpha(t)\Delta\Delta u-\varphi\Delta u\right).
\end{align*}
It follows
\begin{align*}
(\Delta-\partial_t)(uH)&=a(t)\left(2\Delta\Gamma(\sqrt{u})-4\Gamma\left(\sqrt{u},\frac{\Delta u}{2\sqrt{u}}\right)+\alpha'(t)\Delta u+\varphi'u\right)\\
&\hskip 12pt -a'(t)\left(2\Gamma(\sqrt{u})-\alpha(t)\Delta u\right)\\
&=4a(t)\widetilde{\Gamma}_2(\sqrt{u})+(a\cdot\alpha)'\Delta u-2a'(t)\Gamma(\sqrt{u})+(a\varphi)'u,
\end{align*}
which is the desired result.
\end{proof}

Now let us to
\begin{proof}[Proof of Theorem \ref{LY1}] As in the above lemma, let  \begin{equation}\label{def} H(x,t)=a(t)\left(\frac{2\Gamma(\sqrt{u})-\alpha(t)\Delta u}{u}-\varphi(t)\right),\end{equation}
 where the functions $\alpha,\varphi$ are defined in \eqref{notation1}. Fix an arbitrary $T>0$. Our goal is to show that: $H(x,t)\le 0$ holds  for all $ x\in V, t\in [0,T]$. To this end, we divide it into three cases.

\begin{description}
\item{(a).} For $(x,t)\in S_T^c$ and $t>0$, that is $\Delta\sqrt{u}(x,t)\ge0$. Applying\eqref{id12}, it follows $\Delta u(x,t)\ge 2\Gamma(\sqrt{u})(x,t)$, so
    $$
    H(x,t)\le a(t)\left(\frac{2(1-\alpha(t))\Gamma(\sqrt{u})}{u}-\varphi(t)\right).
    $$
Since $\alpha(t)\ge1$, $a(t)>0$ and $\varphi(t)>0$, we have $H(x,t)\le 0$ for $(x,t)\in S_T^c$ and  $t>0$.
\item{(b).} For $(x,t)\in S_T$, that is $\Delta\sqrt{u}(x,t)<0$. By \eqref{diff1} in Lemma \ref{lem1}, we have
    $$
(\Delta-\partial_t)(uH)=4a(t)\widetilde{\Gamma}_2(\sqrt{u})+(a\cdot\alpha)'\Delta u-2a'(t)\Gamma(\sqrt{u})+(a \varphi)' u.
$$

Applying the curvature dimension condition $CDE(n,-K)$, the positivity of $a$ gives
\begin{align*}
(\Delta-\partial_t)(uH)&\ge \frac{4au }{n} \left(\frac{\Delta u-2\Gamma(\sqrt{u})}{2u}\right)^2-4Ka \Gamma(\sqrt{u})+(a\cdot \alpha)'\Delta u-2a' \Gamma(\sqrt{u})+(a \varphi)'u\\
&\ge -\frac{2a\eta}{n}\left(\Delta u-2\Gamma(\sqrt{u})\right)-\frac{au\eta^2}{n}-\left(4Ka+2a'\right)\Gamma(\sqrt{u})+(a\cdot \alpha)'\Delta u+(a \varphi)'u\\
&=\left((a\alpha)'-\frac{2a\eta}{n}\right)\Delta u+\left(-4Ka-2a'+\frac{4a\eta}{n}\right)\Gamma(\sqrt{u})+\left((a\varphi)' -\frac{a\eta^2}{n}\right)u.
\end{align*}

Now let us choose $\eta=\frac{na'}{2a}+nK$, we find that $\alpha,\varphi$ defined in \eqref{notation1} satisfy
\begin{gather}
(a\alpha)'-\frac{2a\eta}{n}=0,\label{diff2}\\
-4Ka-2a'+\frac{4a\eta}{n}=0,\tag{\ref{diff2}$'$}\\
(a\varphi)' -\frac{a\eta^2}{n}=0.\tag{\ref{diff2}$''$}
\end{gather}
Hence we have $(\Delta-\partial_t)(uH)(x,t)\ge0$ for $(x,t)\in S_T$.
\item{(c).} For any $x\in V$, it is easy to see that $H(x,0)=\lim_{t\to0}H(x,t)=0$.
\end{description}

Thanks to  Proposition \ref{mp3}, we have, for all $x\in V,t\in [0,T]$, $u(x,t)H(x,t)\le0$. Hence $H(x,t)\le0$, $\forall x\in V,x\in [0,T]$, thus we complete the proof.   \end{proof}

\begin{remark}
(1). The functions $\alpha,\varphi$ defined in \eqref{notation1} are uniquely determined by the above Eq. (\ref{diff2}), (\ref{diff2}$'$) and (\ref{diff2}$''$).

(2). From the above proof, we have, for $(x,t)$ satisfying $\Delta\sqrt{u}(x,t)<0$,
\begin{equation}\label{diff3}
(\Delta-\partial_t)(uH)(x,t)\ge \frac{au}{n}\left(\frac{2\Gamma(\sqrt{u})-\Delta u}{u}-\frac{na'}{2a}-nK\right)^2.
\end{equation}
\end{remark}

\subsection{Local estimates}
To state the local gradient estimate, let us introduce an additional condition:

{\bf Condition A} There exists some  positive function $\beta$ satisfying 
\begin{equation}\label{bound1}
 \hskip -12pt\frac{\beta'(t)}{\beta(t)}\le\frac{1}{\alpha^2(t)a^2(t)}\left(2k\alpha
a'\int_0^tads+\frac{a}2\int_0^t\frac{a'^2}{a}ds
-2k^2a\int_0^tads\right) \ \mbox{holds for all }\ t>0,
\end{equation}
 in addition, for any fixed $T>0$, there exists some finite, positive function $\eta(T)$ such that $\beta(t)\le \eta(T)$ holds for all $0<t\le T$.

\begin{theorem}\label{local} Let $G=(V,E)$ be a finite (or infinite) graph and $R>0$, and fix $x_0\in V$. Let $u:V\times \mathbb{R}\to \mathbb{R}^+$ a positive function such that $\partial_tu(x,t)=\Delta u(x,t)$, if $d(x,x_0)\le 2R$. If $G$ satisfies the exponential curvature dimension inequality $CDE(n,-K)$ ($K\ge0$), then for $t>0$,
\begin{align}\label{LY-2}
\frac{\Gamma(\sqrt{u})}{u}-\alpha(t)\frac{\partial_t (\sqrt{u})}{\sqrt{u}}-\frac12\varphi(t) \le \frac{nD_{\mu}(1+D_{\omega})\alpha^2(t)\eta(t) }{R\beta(t)}
\end{align}
holds in the ball of radius $R$ around $x_0$, where $\alpha(t),\varphi(t)$ are defined in \eqref{notation1}. Letting $R\to\infty$ gives \eqref{LY-1}.
\end{theorem}
Before the proof, let us give some examples.
\begin{example}
\begin{description}
\item{(1). } Taking $a(t)=t^{\gamma}$ with $1<\gamma<3$. In this case
$\alpha(t)=1+\frac{2Kt}{1+\gamma}$, $\varphi(t)=nK+\frac{nK^2\gamma}{(1+\gamma)}+\frac{n\gamma^2}{4(\gamma-1)t}$, and we can choose $\beta(t)=\exp\left(\int_1^t\frac{(\gamma+\frac{2(\gamma-1)Ks}{1+\gamma})^2}
{2s(\gamma-1)(1+\frac{2Ks}{1+\gamma})^2}ds\right),\eta(t)=\beta(t).$ Hence
 \eqref{LY-2} reduces to
\begin{align*}
\frac{\Gamma(\sqrt{u})}{u}-\left(1+\frac{2Kt}{1+\gamma}\right)\frac{\partial_t (\sqrt{u})}{\sqrt{u}}-\frac12\left(nK+\frac{nK^2\gamma}{(1+\gamma)}+\frac{n\gamma^2}{4(\gamma-1)t}\right) \le \frac{nD_{\mu}(1+D_{\omega})}{R}\left(1+\frac{2Kt}{1+\gamma}\right)^2.
\end{align*}
\item{(2). } Take $a(s)=\sinh^2(Ks)$. In this case,
$$
\alpha(t)=1+\frac{\sinh(Kt)\cosh(Kt)-Kt}{\sinh^2(Kt)},
\varphi(t)=nK(1+\coth(Kt)).
$$
As in \cite{Qian}, we can choose$$ \beta(t)=\tanh(Kt), \eta(t)=\beta(t),
$$
such that {\bf Condition A} holds. Hence \eqref{LY-2} reduces to
\begin{align*}\frac{\Gamma(\sqrt{u})}{u}&-\left(1+\frac{\sinh(Kt)\cosh(Kt)-Kt}{\sinh^2(Kt)}\right)\frac{\partial_t (\sqrt{u})}{\sqrt{u}}-\frac{nK}{2}(1+\coth(Kt))\\
 &\le\frac{nD_{\mu}(1+D_{\omega})}{R}\left(1+\frac{\sinh(Kt)\cosh(Kt)-Kt}{\sinh^2(Kt)}\right)^2.
\end{align*}

\item{(3). } Take
$a(t)=e^{-\frac{2Kt}{1+\beta}}\left(1-e^{-\frac{2Kt}{1+\beta}}\right)^\beta$ with  $\beta\in(1,2]$ for $t<\frac{(1+\beta)\log (1+\beta)}{2K}$. In this case, 
$$
\alpha(t)=e^{\frac{2Kt}{1+\beta}},\
\varphi(t)=\frac{nK\beta^2}{2(\beta-1)(\beta+1)}\frac{e^{\frac{4Kt}{\beta+1}}}{e^{\frac{2Kt}{\beta+1}}-1},
$$
As in \cite{Qian}, we can choose
$$
\beta(t)=e^{-\frac{2Kt}{1+\beta}}\left(1-e^{-\frac{2Kt}{1+\beta}}\right)^{\beta},
\eta(t)=\beta(t),
$$
hence \eqref{LY-2} reduces to \begin{equation*}
\frac{\Gamma(\sqrt{u})}{u}-e^{\frac{2Kt}{1+\beta}}\frac{\partial_t(\sqrt{u})}{\sqrt{u}}-
\frac{nK\beta^2}{4(\beta-1)(\beta+1)}\frac{e^{\frac{4Kt}{\beta+1}}}{e^{\frac{2Kt}{\beta+1}}-1}\le \frac{nD_{\mu}(1+D_{\omega}) }{R}e^{\frac{4kt}{1+\beta}}
\end{equation*}
where $ t<\frac{(1+\beta)\log (1+\beta)}{2K}$.
If the curvature dimension inequality $CDE(n,K)$ with $K\ge0$, we can take  $ a(t)=e^{\frac{2Kt}{1+\beta} }\left(e^{\frac{2Kt}{1+\beta}}-1\right)^{\beta}$,  with  $\beta\in(1,2]$. In this case, we can choose $\beta(t)=\eta(t)=e^{\frac{2Kt}{1+\beta} }\left(e^{\frac{2Kt}{1+\beta}}-1\right)^{\beta}$. Hence \eqref{LY-2} reduces to
\begin{equation*}
\frac{\Gamma(\sqrt{u})}{u}-e^{-\frac{2Kt}{1+\beta}}\frac{\partial_t(\sqrt{u})}{\sqrt{u}}-
\frac{nK\beta^2}{4(\beta-1)(\beta+1)}\frac{e^{-\frac{4Kt}{\beta+1}}}{1-e^{-\frac{2Kt}{\beta+1}}}
\le \frac{nD_{\mu}(1+D_{\omega}) }{R}e^{-\frac{4Kt}{1+\beta}}, \ \forall t>0.
\end{equation*}
\end{description}
\end{example}

Now let us to
\begin{proof}[Proof of Theorem \ref{local}]
Consider the following cut-off function $\phi$, see \cite{BHLLMY},  defined as
$$
\phi(x)=
\begin{cases}
0&: \ \ d(x,x_0)> 2R;\\
\frac{2R-d(x,x_0)}{R}&:\ \  R\le d(x,x_0) \le 2R;\\
1&:\ \ d(x,x_0)<R.
\end{cases}
$$

Let  $$G=\beta(t)\phi\left(\frac{2\Gamma(\sqrt{u})-\alpha(t)\Delta u}{u}-\varphi(t)\right)
=\beta(t)\phi\left(\frac{2(1-\alpha)\Gamma(\sqrt{u})-2\alpha(t)\sqrt{u}\Delta\sqrt{u}}{u}-\varphi(t)\right),$$ where $\beta(t)$ satisfies {\bf Condition A}, and let $(x^*,t^*)$ be the place of where $G$ attains its maximum in $V\times [0,T]$ for any arbitrary but fixed $T>0$. Without loss of any generality, we can assume $G(x^*,t^*)>0$, otherwise there is nothing to prove. It follows $t^*>0, \phi(x^*)>0$ and $\Delta \sqrt{u}(x^*,t^*)<0$. To prove the desired result, let us divide into two cases.

\begin{description}
\item{Case (a). $\phi(x^*)=\frac{1}{R}$, that is $d(x_0,x^*)=2R-1$. } Notice that
    \begin{align*}
    G(x^*,t^*)&=\beta(t^*) \phi(x^*)\left(\frac{2(1-\alpha)\Gamma(\sqrt{u})-2\alpha(t)\sqrt{u}\Delta\sqrt{u}}{u}-\varphi(t)\right){(x^*,t^*)}\\
    &\le 2\alpha(t^*)\beta(t^*)\phi(x^*)\frac{-\Delta\sqrt{u}(x^*,t^*)}{\sqrt{u}(x^*,t^*)}.
    \end{align*}

Since positivity of $u$ implies that for any vertex $x\in V$,
$$
\frac{-\Delta\sqrt{u}(x)}{\sqrt{u}(x)}=\frac{1}{\mu(x)}\sum_{y\sim x}\omega_{xy}\left(1-\frac{\sqrt{u}(y)}{\sqrt{u}(x)}\right)\le \frac{deg(x)}{\mu(x)}\le D_{\mu},
$$
hence we have
\begin{align*}
G(x,T)\le G(x^*,t^*)\le \frac{2D_{\mu}\alpha(t^*)\beta(t^*)}{R}\le \frac{2D_{\mu}\alpha(T)\eta(T)}{R}.
\end{align*}

\item{Case (b). $\phi(x^*)\ge\frac{2}{R}$, that is $d(x_0,x^*)\le 2R-2$.}
In what follows, all computation are understood at the point $(x^*,t^*)$. Applying Lemma 4.1 in \cite{BHLLMY}  with the case of $F=u/\phi$ and $H=G$, we have
\begin{equation*}
(\Delta-\partial_t)\left(\frac{u}{\phi}\right)G\ge (\Delta-\partial_t)\left(\frac{uG}{\phi}\right) =(\Delta-\partial_t)\left(\frac{\beta }{a}uH\right),
\end{equation*}
where $H$ is defined by \eqref{def}. Applying \eqref{diff3}, we have
\begin{align*}
(\Delta-\partial_t)\left(\frac{\beta }{a}uH\right)&=\frac{\beta}{a}(\Delta-\partial_t)(uH)+uH(\Delta-\partial_t)\left(\frac{\beta}{a}\right)\\
&\ge \frac{\beta u}{n}\left(\frac{2\Gamma(\sqrt{u})-\Delta u}{u}-\frac{na'}{2a}-nK\right)^2+\left(\frac{a'}{a}-\frac{\beta'}{\beta}\right)\frac{uG}{\phi}\\
&=\frac{\beta u}{n\alpha^2}\left(\frac{G}{\beta\phi}+2(\alpha-1)\frac{\Gamma(\sqrt{u})}{u}+\varphi-\frac{n\alpha a'}{2a}-nK\alpha\right)^2+\left(\frac{a'}{a}-\frac{\beta'}{\beta}\right)\frac{uG}{\phi}\\
&\ge \frac{uG^2}{n\alpha^2\beta \phi^2}+\frac{2uG}{n\alpha^2\phi}\left(\varphi-\frac{n\alpha a'}{2a}-nK\alpha\right)+\left(\frac{a'}{a}-\frac{\beta'}{\beta}\right)\frac{uG}{\phi}.
\end{align*}
Notice that
\begin{equation}
\frac{a'}{a}+\frac{2}{n\alpha^2}\left(\varphi-\frac{n\alpha a'}{2a}-nK\alpha\right)=\frac{1}{a^2\alpha^2}\left(2K\alpha a'\int_0^ta(s)ds+\frac{a}{2}\int_0^t\frac{a'^2}{a}ds-2K^2a\int_0^tads\right),
\end{equation}
Putting the above inequalities together,  by condition \eqref{bound1}, we have
$(\Delta-\partial_t)\left(\frac{u}{\phi}\right)G\ge \frac{uG^2}{n\alpha^2\beta \phi^2}$, hence
\begin{align*}
G(x^*,t^*)&\le n\alpha^2(t^*)\beta(t^*) \frac{\phi^2(x^*)}{u(x^*,t^*)}(\Delta-\partial_t)\left(\frac{u}{\phi}\right)(x^*,t^*)\\
&\le \frac{2nD_{\mu}D_{\omega}\alpha^2(t^*)\beta(t^*) }{R},
\end{align*}
where the last inequality follows from the fact that
\begin{equation*}
\frac{1}{u(x^*,t^*)}(\Delta-\partial_t)\left(\frac{u}{\phi}\right)(x^*,t^*)
\le\frac{2D_{\mu}D_{\omega}}{R},
\end{equation*}
which have been observed in  \cite{BHLLMY}. Thus,
\begin{align*}
G(x,T)\le G(x^*,t^*)\le \frac{2nD_{\mu}D_{\omega}\alpha^2(T)\eta(T) }{R}.
\end{align*}
\end{description}
Combining the above two cases and the fact $\alpha\ge1$, we have for any $T>0$
$$
G(x,T)\le \frac{2nD_{\mu}(1+D_{\omega})\alpha^2(T)\eta(T) }{R},
$$
Dividing by $\beta(T)$ in the both sides, the arbitrariness of $T$ gives the desired result.\end{proof}

Furthermore, if there exists a strong cut-off function, see definition 5 in
\cite{BHLLMY}, we can show a similar version of Theorem \ref{local} holds with $1/R^2$ instead of $1/R$, see Appendix.

\section{Hamilton type gradient estimates}\label{sec-Ham}
In this section, we will prove a version of Hamilton type gradient estimate for the solutions to heat equation \eqref{heat} on graph $G$.

First let us recall the Hamilton type gradient estimate on Riemannian manifolds.
Let $M$ be a connected Riemannian manifold of dimension
$n$ and $\Delta$ be the Laplace-Beltrami operator on $M$.
Suppose $u(t,\cdot)$ is the positive solution to the  heat equation $
\Delta u(t,\cdot)=\frac{\partial}{\partial t}u(t,\cdot)$
where the initial heat $u(0,\cdot)$.
Assume the Ricci curvature is bounded below, i.e.
$Ricci\ge -K$ for some constant $K\ge0$, Hamilton
\cite{Ham} obtained the following gradient estimate
on compact Riemannian manifolds $M$:

{\bf Theorem } Assume that the solution $u$ to the heat equation is bounded, i.e.   $u\le A$ for $A$ is some positive constant, we have
\begin{equation}\label{hamil0}
|\nabla \log u|^2\le \frac{(1+2Kt)}{t}\log \frac{A}{u}.
\end{equation}

Let us first state the following lemma:
\begin{lemma}\label{zero}
Assume that for the positive solution $u(x,t)$ to the heat equation \eqref{heat} and $|\Delta u|(x,t)$ is differentiable for $t$ in the set of $\Delta u(x,t)=0$, then we have
$$
(\Delta-\partial_t)|\Delta u|(x_0,t_0)\ge0,
$$
for the point $(x_0, t_0)$ satisfying $\Delta u(x_0,t_0)=0$. Hence $(\Delta-\partial_t)|\Delta u|(x,t)\ge0$ holds for any $(x,t)\in V\times \mathbb{R}$.
\end{lemma}
\begin{proof}
Since $\Delta u(x_0,t_0)=0$, it follows $\sum_{y\sim x_0}\omega_{x_0y}u(y,t_0)=\deg(x_0)u(x_0,t_0)$ and $\partial_tu(x_0,t_0)=0$. Through direct computation, we have

\begin{align*}
\partial_t|\Delta u|(x_0,t_0)&=\frac{1}{\mu(x_0)}\limsup_{t'\to t_0}\frac{\left|\sum_{y\sim x_0}\omega_{x_0y}(u(y,t')-u(x_0,t'))\right|}{t'-t_0}\\
&\le \frac{1}{\mu(x_0)}\lim_{t'\to t_0}\left|\sum_{y\sim x_0}\omega_{x_0y}\frac{(u(y,t')-u(y,t_0))-(u(x_0,t')-u(x_0,t_0))}{t'-t_0}\right|\\
&=\frac{1}{\mu(x_0)}\left|\sum_{y\sim x_0}\omega_{x_0y}\partial_tu(y,t_0)\right|,
\end{align*}
thus we obtain
\begin{align*}
(\Delta-\partial_t)|\Delta u|(x_0,t_0)&=\frac{1}{\mu(x_0)}\left(\sum_{y\sim x_0}\omega_{x_0y}|\Delta u|(y,t_0)-\partial_t|\Delta u|(x_0,t_0)\right)\\
&\ge \frac{1}{\mu(x_0)}\left(\sum_{y\sim x_0}\omega_{x_0y}|\Delta u|(y,t_0)-\left|\sum_{y\sim x_0}\omega_{x_0y}\partial_tu(y,t_0)\right|\right)\\
&\ge \frac{1}{\mu(x_0)}\left(\left|\sum_{y\sim x_0}\omega_{x_0y}\Delta u(y,t_0)\right|-\left|\sum_{y\sim x_0}\omega_{x_0y}\partial_tu(y,t_0)\right|\right)=0.
\end{align*}
The desired result follows.\end{proof}

Now let us state the main results in this section.
\begin{theorem}\label{hamilton1}
Let $G=(V,E)$ be a finite (or infinite) graph satisfying the curvature dimension condition $CDE(\infty,-K)$ with $K\ge0$, and let $u$ be a positive solution to the heat equation \eqref{heat} on $G$. Assume that $|\Delta u|(x,t)$ is differentiable for $t$ in the set of $\Delta u(x,t)=0$ and $u\le A$ for some positive constant $A$,  then for all $t>0$,
\begin{equation}\label{Ham-12}
\Gamma(\sqrt{u})\le \frac12 \left|\Delta u \right|+\frac{(1+2Kt)\sqrt{A}}{t}\sqrt{u}.
\end{equation}
\end{theorem}
\begin{proof}
Denote $\varphi(t)=\frac{t}{1+2Kt}$, it is easy to see that $\varphi'(t)\ge0$ and \begin{equation*}\label{diff-4}\varphi'(t)+2K\varphi(t)\le 1.\end{equation*}

Denote $$
H=\varphi(t)\Gamma(\sqrt{u})-\frac12\varphi(t)\left|\Delta u\right|-\sqrt{A}\sqrt{u},
$$
for an arbitrary $T>0$, our goal is to show $H\le 0$ holds for all $x\in V, t\in [0,T]$. To this end, as in the proof of Theorem \ref{LY1}, we divide into two cases.
\begin{description}
\item{(a). For $(x,t)\in S_T^c$, i.e. $\Delta\sqrt{u}(x,t)\ge0.$} Applying \eqref{id12}, we have $\Gamma(\sqrt{u})\le \frac12\Delta u$. Consequencely $H\le 0$ holds for $(x,t)\in S_T^c$.

\item{(b). For $(x,t)\in S_T$, i.e. $\Delta\sqrt{u}(x,t)<0.$ or equivalently $\Delta u(x,t)<2\Gamma(\sqrt{u})(x,t)$.}
 Let $\Phi(t)=\varphi(t)\Gamma(\sqrt{u})$, we have, by applying the curvature dimension condition $CDE(\infty,-K)$,
\begin{align*}
(\Delta-\partial_t)\Phi(t)&=\varphi(t)\Delta\Gamma(\sqrt{u})-\varphi'(t)\Gamma(\sqrt{u})-2\varphi(t)\Gamma\left(\sqrt{u},\frac{\Delta u}{2\sqrt{u}}\right)\\
&=2\varphi(t)\widetilde{\Gamma}_2(\sqrt{u})-\varphi'(t)\Gamma(\sqrt{u})\\
&\ge -\left(\varphi'(t)+2K\varphi(t)\right)\Gamma(\sqrt{u})\\
&\ge -\Gamma(\sqrt{u}).
\end{align*}

Notice that, by \eqref{id12}, $$
(\Delta -\partial_t)\sqrt{u}=\frac{\Delta u-\partial_tu}{2\sqrt{u}}-\frac{\Gamma(\sqrt{u})}{\sqrt{u}}=-\frac{\Gamma(\sqrt{u})}{\sqrt{u}},$$
it follows, for $(x,t)\in S_T$,
\begin{align*}
(\Delta-\partial_t)H(x,t)\ge -\Gamma(\sqrt{u})+\frac{\varphi'(t)}{2}\left|\Delta u\right|+\frac{\sqrt{A}}{\sqrt{u}}\Gamma(\sqrt{u})\ge0,
\end{align*}
where we have used the Lemma \ref{zero}.
\end{description}
 For any $x\in V$, obviously $H(x,0)\le0$. Thanks to Proposition \ref{mp3}, we have $H(x,t)\le 0$ holds for any $x\in V, t\in [0,T]$, thus we complete the proof.
\end{proof}
\begin{remark}
Notice that for $(x,t)$ satisfying $\Delta u(x,t)=0$ for some positive function $u(x,t)$, then we have
$$
\Gamma(\sqrt{u})(x,t)\le \frac{1}{2\mu(x)}\sum_{y\sim x}w_{xy}(u(y)+u(x))=\frac{\deg(x)}{\mu(x)}u(x)\le D_{\mu}u(x).
$$
From the above proof, we see that: Let $G=(V,E)$ be a finite (or infinite)
 graph satisfying the curvature dimension condition $CDE(\infty,-K)$
 with $K\ge0$. Let $u$ be a positive solution to the heat equation
 \eqref{heat} on $G$, and $u\le A$ holds for some positive constant $A$.
 Then we have for all $t>0$,
\begin{equation*}
\Gamma(\sqrt{u})\le \frac12 \left|\Delta u \right|+\left(\frac{1}{t}+(2K)\vee D_{\mu}\right)\sqrt{A}\sqrt{u}.
\end{equation*}
\end{remark}

As an application of Theorem \ref{hamilton1}, we have:

\begin{corollary}\label{lv2}
Let $G=(V,E)$ be a finite (or infinite) graph satisfying the curvature dimension condition $CDE(\infty,0)$, and let $u$ be a positive solution to the equation $\Delta u=0$ on $G$ and $u\le A$ for some positive constant $A$. I.e. there is no bounded positive harmonic functions on $G$.\end{corollary}
\begin{proof}
Applying \eqref{Ham-12}, we have $\Gamma(\sqrt{u})(x)=0$ for any $x\in V$, hence $u$ must be a constant.
\end{proof}
\begin{remark}
The  curvature condition $CDE(\infty,0)$ in  Corollary \ref{lv2} is weaker than the one ($CDE(n,0)$) in Corollary \ref{lv1}. There does exist some graphs $G$ satisfies $CDE(\infty,0)$ but does not satisfies $CDE(n,0)$. For example, in \cite{BHLLMY}, they show that for $d-$regular Ricci-flat graph with   weak consistent weighting, the Laplacian with $\mu(x)\equiv \mu$, $G$ satisfies $CDE(\infty,0)$. Furthermore, if the weighting of $G$ is consistent, then $G$ satisfies $CDE(d,0)$.
\end{remark}

\section{Harnack inequalities and heat kernel estimates}\label{sec-Harn}

In this section, we will derive Harnack-type inequalities as consequences of the above gradient estimates. The proof is following the ones in \cite{LY,BHLLMY}.

Before we state the main result in this section, we need one simple lemma, which is a generalization of Lemma 5.3 in \cite{BHLLMY}.
\begin{lemma}\label{lem-1}
For any functions $\psi,\alpha:[T_1,T_2]\to \mathbb{R}$, we have
\begin{equation*}
\min_{s\in[T_1,T_2]}\left(\psi(s)-\frac{1}{\alpha(s)}\int_s^{T_2}\psi^2(t)dt\right)\le \frac{2\int_{T_1}^{T_2}ds\int_{T_1}^s\alpha(u)du}{(T_2-T_1)^3}.
\end{equation*}
\end{lemma}
\begin{proof}
The method is applying Cauthy-Schwartz inequality and the idea that we can bound the minimum by an averaged sum. Let $\phi(s)=2\int_{T_1}^t\frac{1}{\alpha(s)}ds$, applying the Fubini theorem, on has $\phi^2(t)=4\int_{T_1}^t\frac{\phi(s)}{\alpha(s)}ds$ for any $t\in[T_1,T_2]$.  Applying the Fubini theorem again, one obtains
\begin{align*}
\min_{s\in[T_1,T_2]}\left(\psi(s)-\frac{1}{\alpha(s)}\int_s^{T_2}\psi^2(t)dt\right)&\le \frac{1}{\int_{T_1}^{T_2}\phi(s)ds}\int_{T_1}^{T_2}\phi(s)\left(\psi(s)-\frac{1}{\alpha(s)}\int_{s}^{T_2}\psi^2(t) dt\right)ds\\
&=\frac{1}{\int_{T_1}^{T_2}\phi(s)ds}\int_{T_1}^{T_2}\psi(t)\phi(t)-
\psi^2(t)\int_{T_1}^{s}\frac{\phi(s)}{\alpha(s)}dsdt\\
&\le \frac{1}{\int_{T_1}^{T_2}\phi(s)ds}\int_{T_1}^{T_2}\frac{\phi^2(t)}{4\int_{T_1}^{t}\frac{\phi(s)}{\alpha(s)}ds}dt\\
&=\frac{T_2-T_1}{\int_{T_1}^{T_2}\phi(s)ds}.
\end{align*}

Thanks to Cauchy-Schwartz inequality, we obtain

\begin{equation*}
\frac{1}{\int_{T_1}^{T_2}ds\int_{T_1}^s\frac{1}{\alpha(u)}du}\le \frac{\int_{T_1}^{T_2}ds\int_{T_1}^s\alpha(u)du}{\frac{(T_2-T_1)^4}{4}},
\end{equation*}
 the desired result follows.
\end{proof}

Now let us state the main result in this section.
\begin{theorem}\label{hak}
Let $G=(V,E)$ be a finite (or infinite) graph with measure bound $\mu_{max}$, and suppose that a function $f:V\times \mathbb{R}\to \mathbb{R}$ satisfies
$$
\frac{\Gamma(f)}{f^2}-\alpha(t)\frac{\partial_t f}{f}\le \psi(t)
$$
whenever $x\in B(x_o,R)$ for $x_o\in V$ along with some $R>0$, some functions $\alpha(t)\ge1, \psi$. Then for  $T_1<T_2$ and $x,y\in V$, we have
$$
f(x,T_1)\le f(y,T_2)\exp{\left(\int_{T_1}^{T_2}\frac{\psi(t)}{\alpha(t)}dt+\rho_{\mu_{max},\alpha,\omega_{min}}(x,y,T_1,T_2)\right)}
$$
where $$
\rho_{\mu_{max},\alpha,\omega_{min}}(x,y,T_1,T_2)=\inf\left\{\sum_{i=0}^{k-1}\frac{4\mu_{max}k^3}{\omega_{min}(T_2-T_1)^3}
\int_{t_{i}}^{t_{i+1}}ds\int_{t_{i}}^s \alpha(u)du\right\}
$$
the infinium is taken over the set of all paths $P=p_0p_1\cdots p_k$ so that $x_0=x,x_k=y$ and having all $p_i\in B(x_o,R)$, and the times $T_1=t_0,t_1,\cdots t_k=T_2$ evenly divide the interval $[T_1,T_2]$.
\end{theorem}
\begin{proof} The proof here follows exactly the one of Theorem 5.1 in \cite{BHLLMY}, we write it for the readers' convenience.  Let us first assume that $x\sim y$. Then for any $s\in [T_1,T_2]$ we can write
\begin{align*}
\log f(x,T_1)-\log f(y,T_2)=-\int_{T_1}^s\frac{\partial}{\partial t}\log f(x,t)dt+\log\frac{f(x,s)}{f(y,s)}-\int_s^{T_2}\frac{\partial}{\partial t}\log f(y,t)dt.
\end{align*}
Applying the assumption
$$
-\frac{\partial}{\partial t}\log f\le -\frac{1}{\alpha(t)}\frac{\Gamma(f)}{f^2}+\frac{\psi(t)}{\alpha(t)},
$$
we have
\begin{align*}
\log f(x,T_1)-\log f(y,T_2)&\le\int_{T_1}^{T_2}\frac{\psi(t)}{\alpha(t)}dt
-\int_{T_1}^s\frac{1}{\alpha(t)}\frac{\Gamma(f)}{f^2}(x,t)dt
-\int_{s}^{T_2}\frac{1}{\alpha(t)}\frac{\Gamma(f)}{f^2}(y,t)dt+\log\frac{f(x,s)}{f(y,s)}\\
&\le \int_{T_1}^{T_2}\frac{\psi(t)}{\alpha(t)}dt-\frac{\omega_{min}}{2\mu_{max}}
\int_{s}^{T_2}\frac{1}{\alpha(t)}\left|\frac{f(x,t)-f(y,t)}{f(y,t)}\right|^2dt+\frac{f(x,s)-f(y,s)}{f(y,s)},
\end{align*}
where in the second step we threw way the $\int_{T_1}^{s}$ term, and used that $\Gamma(f)(y,t)\ge\frac12\omega_{min}(f(y,t)-f(x,t))^2/\mu_{max}$ as well as the fact that $\log r\le r-1$ for any $r>0$.
 Using Lemma \ref{lem-1}, we have
 \begin{align*}
 \log f(x,T_1)-\log f(y,T_2)\le \int_{T_1}^{T_2}\frac{\psi(t)}{\alpha(t)}dt+ \frac{4\mu_{max}}{\omega_{min}(T_2-T_1)^3}
 \int_{T_1}^{T_2}ds\int_{T_1}^s \alpha(u)du.
 \end{align*}
 For $x,y\in B(x_o,R)$, there exists a path $P=x,x_1,\cdots,x_k=y$ such that $x:=x_0\sim x_1$ and $x_i\sim x_{i+1}$ for $1\le i\le k-1$ as well as  $x_i\in B(x_o,R)$. Let $T_1=t_0<\cdots<t_{k}=T_2$ be a subdivision of the time interval $[T_1,T_2]$ into $k$ equal parts.  We have
  \begin{align*}
 \log f(x,T_1)-\log f(y,T_2)&=\sum_{i=0}^{k-1}\log f(x_i,t_i)-\log f(x_{i+1},t_{i+1})\\
  &\le \int_{T_1}^{T_2}\frac{\psi(t)}{\alpha(t)}dt+ \sum_{i=0}^{k-1}\frac{4\mu_{max}k^3}{\omega_{min}(T_2-T_1)^3}
\int_{t_{i}}^{t_{i+1}}ds\int_{t_{i}}^s \alpha(u)du.
 \end{align*}
Minimizing all the path, we complete the proof.
\end{proof}
\begin{remarks}\label{rem2}
(1). Take $\alpha(t)=1+\frac{2Kt}{1+\gamma}$, for some $\gamma\in (1,3)$, we have
\begin{equation*}
\rho_{\mu_{max},\alpha,\omega_{min}}(x,y,T_1,T_2)\le \frac{2\mu_{max}d^2(x,y)}{\omega_{min}(T_2-T_1)}\left(1+\frac{K(T_2+T_1)}{1+\gamma}\right).
\end{equation*}

(2). Take $\alpha(t)=1+\frac{\sinh(Kt)\cosh(Kt)-Kt}{\sinh^2(Kt)}$,  we have
 \begin{equation*}
\rho_{\mu_{max},\alpha,\omega_{min}}(x,y,T_1,T_2)\le \frac{2\mu_{max}d^2(x,y)}{\omega_{min}(T_2-T_1)}\left(1+\coth(KT_1)\right).
\end{equation*}
Moreover, if $T_1<T_2< T_1\big(d(x,y)+1\big)$, we have
\begin{equation*}
\rho_{\mu_{max},\alpha,\omega_{min}}(x,y,T_1,T_2)\le \frac{2\mu_{max}d^2(x,y)}{\omega_{min}(T_2-T_1)}\left(1+\frac{1}{K}\ln \frac{\sinh(KT_2)}{\sinh\left(KT_1-\frac{K(T_2-T_1)}{d(x,y)}\right)}\right).
\end{equation*}

If  $0<KT_2<\delta$ for some positive constant $0<\delta=\delta(T_1,T_2,K)<1$, we have

 \begin{equation*}
\rho_{\mu_{max},\alpha,\omega_{min}}(x,y,T_1,T_2)\le \frac{2(1+KT_2)\mu_{max}d^2(x,y)}{\omega_{min}(T_2-T_1)}.
\end{equation*}

(3). Take $\alpha(t)=e^{\frac{2Kt}{1+\beta}}$ with $\beta\in (1,2]$ and $K\ge0$ for $t<\frac{(1+\beta)\log (1+\beta)}{2K}$, we have for $0<T_1<T_2<\frac{(1+\beta)\log (1+\beta)}{2K}$,
\begin{equation*}
\rho_{\mu_{max},\alpha,\omega_{min}}(x,y,T_1,T_2)\le \frac{2\mu_{max}d^2(x,y)}{\omega_{min}(T_2-T_1)}
\left(1+\frac{1+\beta}{2K}e^{\frac{T_2-T_1}{d(x,y)}}\left(e^{\frac{2K}{1+\beta}T_2}
-e^{\frac{2K}{1+\beta}T_1}\right)-(T_2-T_1)\right).
\end{equation*}

We shall give the proof in the Appendix.
\end{remarks}

As an application of the above theorem, we have the Harnack inequalities as follows.
\begin{corollary}\label{harnack}
Suppose $G=(V,E)$ is a finite or infinite unweighted graph satisfying $CDE(n,-K)$ with $K\ge0$ and $\mu(x)=\deg(x)$ for all $x\in V$. Denote by $D$  the maximum degree of a vertex in G. If $u$ is a positive solution to the heat equations on $G$, then for $0<T_1<T_2$,
$$
u(x,T_1)\le u(y,T_2)\exp{\left(2\int_{T_1}^{T_2}\frac{\psi(t)}{\alpha(t)}dt+2\rho_{\mu_{max},\alpha,\omega_{min}}(x,y,T_1,T_2)\right)}
$$
holds for $x,y\in V$.
In particular, for $0<T_1<T_2<\infty$ and $1<\gamma<3$,
\begin{equation*}
u(x,T_1)\le u(y,T_2)\left(\frac{T_2}{T_1}\right)^{\frac{n\gamma^2}{4(\gamma-1)}}\left(\frac{1+\frac{2K}{1+\gamma}T_2}{1+\frac{2K}{1+\gamma}T_1}\right)^{-\frac{n}{4(\gamma-1)}}
\exp\left(\frac{4Dd^2(x,y)}{T_2-T_1}\left(1+\frac{K(T_2+T_1)}{1+\gamma}\right)+\frac{nK}{2}(T_2-T_1)\right),
\end{equation*}

\begin{equation*}
u(x,T_1)\le u(y,T_2)\left(\frac{e^{2KT_2}-2KT_2-1}{e^{2KT_1}-2KT_1-1}\right)^{\frac{n}2}
\exp\left(\frac{4Dd^2(x,y)}{T_2-T_1}\left(1+\coth(KT_1)\right)\right),
\end{equation*}
hold for $T_2>T_1>0, x,y\in V$.
Moreover, if $KT_2<\delta$ for some positive $0<\delta<1$, we have
\begin{equation*}
u(x,T_1)\le u(y,T_2)\left(\frac{e^{2KT_2}-2KT_2-1}{e^{2KT_1}-2KT_1-1}\right)^{\frac{n}2}
\exp\left(\frac{4Dd^2(x,y)}{T_2-T_1}\left(1+KT_2\right)\right).
\end{equation*}

 If  we further assume  that $K=0$, i.e. $G=(V,E)$ satisfies $CDE(n,0)$.  We have
\begin{equation*}
u(x,T_1)\le u(y,T_2)\left(\frac{T_2}{T_1}\right)^{\frac{n}2}
\exp\left(\frac{4Dd^2(x,y)}{T_2-T_1}\right).
\end{equation*}
\end{corollary}
\begin{proof} Together with  Theorem \ref{LY1}, Theorem \ref{hak} and Remarks \ref{rem2}, we can complete the proof directly.\end{proof}
\begin{remark}\label{rem3}
The above theorem can be seen a generalization of Theorem 5.1 in \cite{BHLLMY}.
\end{remark}
As an application, we have the following estimate for the associated heat kernel.
\begin{theorem}\label{kernel}
Suppose $G=(V,E)$ satisfies $CDE(n,-K)$ ($K\ge0$) and has maximum degree $D$. Denote by $P_t(x,y)$ the fundamental solution (heat kernel) to the heat equation \eqref{heat} starting at $x$.  Then there exist constants $C_1,C_1',C_2,C_2',C_3$ depending on $K,n,D$ so that for $t>1$,
\begin{equation*}
C_1e^{-\frac{nKt}2}t^{-n}\left(1+\frac23Kt\right)^{\frac{n}4}e^{\left(-\frac{4Dd^2(x,y)}{t-1}\left(1+\frac{K}{3}(t+1)\right)\right)}\le P_t(x,y)\le C_2\frac{\mu(y)}{\mbox{vol}(B(x,\sqrt{t}))}e^{3DKt+\frac{nKt}2},
\end{equation*}
\begin{equation*}
P_t(x,y)\le C_2'\frac{\mu(y)}{\mbox{vol}(B(x,\sqrt{t}))}
\left(\frac{e^{4Kt}-4Kt-1}{e^{2Kt}-2Kt-1}\right)^{\frac{n}{2}}e^{4D\coth(Kt)},
\end{equation*}
and
\begin{equation*}
 P_t(x,y)\ge C_1'\left(e^{2Kt}-2Kt-1\right)^{-\frac{n}2}\exp\left(-\frac{C_3d^2(x,y)}{t-1}\right).
\end{equation*}

\begin{proof}
Following the proof of Theorem 9 in \cite{BHLLMY}, we can obtain the desired result easily with the help of Corollary \ref{harnack}.
\end{proof}

\end{theorem}


\section{Appendix}
 Let us recall the strong cut-off function defined in \cite{BHLLMY}.
\begin{definition}
Let $(G,V)$ be a graph satisfying $CDE(n,-K)$ for $K\ge0$. We say that the function $\phi:\ V\to [0,1]$ is an $(c,R)-strong$ cut-off function centered at $x_0\in V$ and supported on a set $S\subset V$ if $\phi(x_0)=1,\ \phi(x)=0$ if $x\notin S$ and for any vertex $x\in S$,  either $\phi(x)\le \frac{c(1+R\sqrt{K})}{2R^2}$ or $\phi$ does not vanish in the immediate neighborhood of $x_0$ and $\phi^2(x)\Delta \frac{1}{\phi}(x)\le D_{\mu}\frac{c(1+R\sqrt{K})}{R^2}$ and $\phi^3(x)\Gamma\left(\frac{1}{\phi}\right)(x)\le D_{\mu}\frac{c}{R^2}$, where $c$ is some positive constant only dependent on $n$.
\end{definition}

The usual Cayley graph of $\mathbb{Z}^d$ with the regular or the normalized Laplacian satisfies satisfies CDE(2d,0) and admits a $(100,R)-strong$ cut-off function supported on a ball  of radius $\sqrt{d}R$ centered at $x_0$, cf. \cite{BHLLMY}.

We have the following theorem:
\begin{theorem}\label{local-2}
Let $(G,V)$ be a (finite or infinite) graph satisfying $CDE(n,-K)$ for $K\ge0$. Let $R>0$ and fix $x_0\in V$. Assume that $G$ has a $(c,R)-strong$ cut-off function supported on $S\subset V$ and centered at $x_0$. Let $u:V\times \mathbb{R}\to \mathbb{R}$ a positive function such that $\partial_tu(x,t)=\Delta u(x,t)$ if $x\in S$, then
\begin{align*}
\frac{\Gamma(\sqrt{u})}{u}(x_0,t)-\alpha(t)\frac{\partial_t (\sqrt{u})}{\sqrt{u}}(x_0,t)-\frac12\varphi(t) \le \frac{ncD_{\mu}(1+R\sqrt{K})\alpha^2(t)\eta(t) }{R^2\beta(t)}+\frac{cn^2D_{\mu}(1+D_{\omega})^2\alpha^4(t)\widetilde{\eta}(t)}{4R^2\beta(t)},
\end{align*}
where $\alpha,\varphi$ are defined in \eqref{notation1} and $\beta$ satisfies
{\bf Condition B}:
(1). Condition A holds.
(2). For any $0<t\le T$, $\frac{\beta(t)}{\alpha(t)-1}\le \widetilde{\eta}(T)$ holds for some positive function $\widetilde{\eta}$.
\end{theorem}
\begin{proof}
The proof is similar to the one of Theorem \ref{LY1}, except that we assume $\phi$ is a $(c,R)-strong$ cut-off function centered at $x_0$. Let
$$G=\beta(t)\phi\left(\frac{2\Gamma(\sqrt{u})-\alpha(t)\Delta u}{u}-\varphi(t)\right)
=\beta(t)\phi\left(\frac{2(1-\alpha)\Gamma(\sqrt{u})-2\alpha(t)\sqrt{u}\Delta\sqrt{u}}{u}-\varphi(t)\right),$$ where $\beta(t)$ satisfies conditions in the above theorem, and let $(x^*,t^*)$ be the place of where $G$ attains its maximum in $V\times [0,T]$ for any arbitrary but fixed $T>0$. Without loss of any generality, we can assume $G(x^*,t^*)>0$, otherwise there is nothing to prove. It follows $t^*>0, \phi(x^*)>0$ and $\Delta \sqrt{u}(x^*,t^*)<0$. To prove the desired result, let us divide into two cases.
\begin{description}
\item{(a). Assume $\phi(x^*)\le \frac{c(1+R\sqrt{K})}{R^2}$.} As in the proof of Theorem \ref{LY1}, we have  \begin{align*}
    G(x^*,t^*)\le 2\alpha(t^*)\beta(t^*)\phi(x^*)\frac{-\Delta\sqrt{u}(x^*,t^*)}{\sqrt{u}(x^*,t^*)},
    \end{align*}
combining the fact that $\frac{-\Delta\sqrt{u}(x)}{\sqrt{u}(x)}\le D_{\mu}$, hence \begin{align*}
G(x,T)\le G(x^*,t^*)\le \frac{c(1+R\sqrt{K})D_{\mu}\alpha(t^*)\beta(t^*)}{R^2}\le \frac{c(1+R\sqrt{K})D_{\mu}\alpha(T)\eta(T)}{R^2}.
\end{align*}

\item{(b). Now we assume $\phi$ does not vanish in the immediate neighborhood of $x_0$ and $\phi^2(x)\Delta \frac{1}{\phi}(x)\le D_{\mu}\frac{c(1+R\sqrt{K})}{R^2}$ and $\phi^3(x)\Gamma\left(\frac{1}{\phi}\right)(x)\le D_{\mu}\frac{c}{R^2}$. } We can seen from the proof of Theorem \ref{LY1} that, at $(x^*,t^*)$,
    \begin{align*}
    (\Delta-\partial_t)\left(\frac{u}{\phi}\right)G\ge \frac{uG^2}{n\alpha^2\beta\phi^2}+\frac{4(\alpha-1)uG}{n\alpha^2\phi}\frac{\Gamma(\sqrt{u})}{u},
    \end{align*}
   it follows $$
    G(x^*,t^*)\le n\alpha^2(t^*)\beta(t^*) \frac{\phi^2(x^*)}{u(x^*,t^*)}(\Delta-\partial_t)\left(\frac{u}{\phi}\right)(x^*,t^*)-
   4(\alpha(t^*)-1)\beta(t^*)\phi(x^*)\frac{\Gamma(\sqrt{u})}{u}(x^*,t^*) $$
Notice that for any $x\in V,t>0$,
\begin{align*}x
\phi^2(x)(\Delta-\partial_t)\left(\frac{u}{\phi}\right)(x,t)&=\phi(x)(\Delta-\partial_t)u(x,t)+u(x,t)\phi^2(x)\Delta\left(\frac{1}{\phi}\right)(x)+2\phi^2(x)\Gamma\left(\frac{1}{\phi},u\right)(x,t)\\
&\le \frac{cD_{\mu}(1+R\sqrt{K})}{R^2}u(x,t)+2\phi^2(x)\sqrt{\Gamma\left(\frac1{\phi}\right)}\sqrt{\Gamma(u)}\\
&\le \frac{cD_{\mu}(1+R\sqrt{K})}{R^2}u(x,t)+\frac{2\sqrt{cD_{\mu}}}{R}\sqrt{\phi(x)\Gamma(u)},
\end{align*}
applying the fact that $\frac{\Gamma(u)}{u^2}\le (D_{\omega}+1)^2\frac{\Gamma(\sqrt{u})}{u}$,  cf. equation (4.25) in \cite{BHLLMY}, we have at $(x^*,t^*)$,
\begin{align*}
G&\le \frac{ncD_{\mu}(1+R\sqrt{K})\alpha^2\beta}{R^2}+\frac{2\sqrt{cD_{\mu}}(D_{\omega}+1)n\alpha^2\beta}{R}\sqrt{\phi\frac{\Gamma(\sqrt{u})}{u}}
-4(\alpha-1)\beta\cdot \phi\frac{\Gamma(\sqrt{u})}{u}\\
&\le \frac{ncD_{\mu}(1+R\sqrt{K})\alpha^2\beta}{R^2}+\frac{n^2cD_{\mu}(1+D_{\omega})^2\alpha^4\beta}{4(\alpha-1)R^2}.
\end{align*}
Hence, \begin{align*}
G(x_0,T)\le G(x^*,t^*)&\le \frac{ncD_{\mu}(1+R\sqrt{K})\alpha^2(t^*)\beta(t^*)}{R^2}+\frac{n^2cD_{\mu}(1+D_{\omega})^2\alpha^4(t^*)\beta(t^*)}{4(\alpha(t^*)-1)R^2}\\
&\le \frac{ncD_{\mu}(1+R\sqrt{K})\alpha^2(T)\eta(T)}{R^2}+\frac{n^2cD_{\mu}(1+D_{\omega})^2\alpha^4(T)\widetilde{\eta}(T)}{4R^2}.
\end{align*}
\end{description}

Dividing by $\beta(T)$ in the both sides, the arbitrariness of $T$ gives the desired result.\end{proof}
Let us give some examples to show  {\bf condition B} of Theorem \ref{local-2} holds.
\begin{example}
\begin{description}
\item{(1). } Taking $a(t)=t^{\gamma}$ with $1<\gamma<3$. In this case
$\alpha(t)=1+\frac{2Kt}{1+\gamma}$, $\varphi(t)=nK+\frac{nK^2\gamma}{(1+\gamma)}+\frac{n\gamma^2}{4(\gamma-1)t}$, and we can choose $\beta(t)=\exp\left(\int_1^t\frac{(\gamma+\frac{2(\gamma-1)Ks}{1+\gamma})^2}
{2s(\gamma-1)(1+\frac{2Ks}{1+\gamma})^2}ds\right),\eta(t)=\beta(t),
\widetilde{\eta}(t)\equiv constant$ such that  {\bf Condition B} is satisfied.
\item{(2). } Take $a(s)=\sinh^2(Ks)$. In this case,
$$
\alpha(t)=1+\frac{\sinh(Kt)\cosh(Kt)-Kt}{\sinh^2(Kt)},
\varphi(t)=nK(1+\coth(Kt)).
$$
As above we can choose$$ \beta(t)=\tanh(Kt), \eta(t)=\beta(t), \widetilde{\eta}(t)=constant,
$$
such that  {\bf Condition B} is satisfied.

\item{(3). } Take
$a(t)=e^{-\frac{2Kt}{1+\beta}}\left(1-e^{-\frac{2Kt}{1+\beta}}\right)^\beta$ with  $\beta\in(1,2]$. In this case, 
$$
\alpha(t)=e^{\frac{2Kt}{1+\beta}},\
\varphi(t)=\frac{nK\beta^2}{2(\beta-1)(\beta+1)}\frac{e^{\frac{4Kt}{\beta+1}}}{e^{\frac{2Kt}{\beta+1}}-1},
$$
As above we can choose
$$
\beta(t)=e^{-\frac{2Kt}{1+\beta}}\left(1-e^{-\frac{2Kt}{1+\beta}}\right)^{\beta},
\eta(t)=\beta(t), \widetilde{\eta}(t)=1,
$$
such that {\bf Condition B} holds, hence Theorem \ref{local-2} holds for  $ t<\frac{(1+\beta)\log (1+\beta)}{2K}$.
If the curvature dimension inequality $CDE(n,K)$ with $K\ge0$, we can take  $ a(t)=e^{\frac{2Kt}{1+\beta} }\left(e^{\frac{2Kt}{1+\beta}}-1\right)^{\beta}$,  with  $\beta\in(1,2]$. In this case, we can choose $\beta(t)=\eta(t)=e^{\frac{2Kt}{1+\beta} }\left(e^{\frac{2Kt}{1+\beta}}-1\right)^{\beta}$, $\widetilde{\eta}(t)=1$ such that {\bf Condition B} holds, hence Theorem \ref{local-2} holds for any $t>0$.
\end{description}
\end{example}

\begin{proof}[Proof of Remark \ref{rem2}]
\item{(1).} Since
\begin{align*}
\int_{t_i}^{t_{i+1}}ds\int_{t_i}^{s}1+\frac{2Ku}{1+\gamma}du
=\frac{(t_{i+1}-t_i)^2}{2}+\frac{K}{3(1+\gamma)}(t_{i+1}-t_i)^2
\left(t_{i+1}+2t_i\right)
\end{align*}
if follows
\begin{align*}
\sum_{i=0}^{k-1}\int_{t_i}^{t_{i+1}}ds\int_{t_i}^{s}1+\frac{2Ku}{1+\gamma}du
&=\frac{(T_2-T_1)^2}{2k}+\frac{K(T_2-T_1)^2}{3(1+\gamma)k^2}\left(3kT_1+\frac{T_2-T_1}{k}\sum_{i=0}^{k-1}(3i+1)\right)\\
&\le \frac{(T_2-T_1)^2}{2k}+\frac{K(T_2-T_1)^2(T_2+T_1)}{2(1+\gamma)k}.
\end{align*}
Hence,
\begin{equation*}
\rho_{\mu_{max},\alpha,\omega_{min}}(x,y,T_1,T_2)\le \frac{2\mu_{max}k^2}{\omega_{min}(T_2-T_1)}\left(1+\frac{K(T_2+T_1)}{1+\gamma}\right),
\end{equation*}
for all $k\ge d(x,y)$, the desired result follows.
\item{(2).} Since
\begin{align*}
\int_{t_i}^{t_{i+1}}ds\int_{t_i}^s1+\frac{\sinh(Ku)\cosh(Ku)-Ku}{\sinh^2(Ku)}du&=
\int_{t_i}^{t_{i+1}}(s-t_i)+s\coth(Ks)-t_i\coth(Kt_i) ds\\
&\le \frac{(T_2-T_1)^2}{2k^2}+\coth(Kt_i)\int_{t_i}^{t_{i+1}}s-t_ids\\
&=\frac{(T_2-T_1)^2}{2k^2}\left(1+\coth(Kt_i)\right),
\end{align*}
it follows
\begin{align*}
\sum_{i=0}^{k-1}\int_{t_i}^{t_{i+1}}ds\int_{t_i}^s1+\frac{\sinh(Ku)\cosh(Ku)-Ku}{\sinh^2(Ku)}du\le \frac{(T_2-T_1)^2}{2k}\left(1+\frac{\sum_{i=0}^{k-1}\coth(Kt_i)}{k}\right).\end{align*}
Since $\coth(Kx)$ is decreasing, hence
\begin{align*}
\sum_{i=0}^{k-1}\int_{t_i}^{t_{i+1}}ds\int_{t_i}^s1+\frac{\sinh(Ku)\cosh(Ku)-Ku}{\sinh^2(Ku)}du\le \frac{(T_2-T_1)^2}{2k}\left(1+\coth(Kt_1)\right).
\end{align*}
Moreover, if $T_1<T_2< T_1\big(d(x,y)+1\big)$,  then $t_{-1}>0$, we have
\begin{align*}
\sum_{i=0}^{k-1}\int_{t_i}^{t_{i+1}}ds\int_{t_i}^s1+\frac{\sinh(Ku)\cosh(Ku)-Ku}{\sinh^2(Ku)}du&\le \frac{(T_2-T_1)^2}{2k}\left(1+\int_{t_{-1}}^{t_k}\coth(Kx)dx\right)\\
&\le \frac{(T_2-T_1)^2}{2k}\left(1+\frac{1}{K}\ln \frac{\sinh(KT_2)}{\sinh K\left(T_1-\frac{T_2-T_1}{k}\right)}\right)
\end{align*}
for all $k\ge d(x,y)$, the desired result follows.

If $KT_2<\delta$ for some constant $1>\delta>0$, sine
 since
\begin{align*}
\int_{t_i}^{t_{i+1}}ds\int_{t_i}^s1+\frac{\sinh(Ku)\cosh(Ku)-Ku}{\sinh^2(Ku)}du&=
\int_{t_i}^{t_{i+1}}(s-t_i)+s\coth(Ks)-t_i\coth(Kt_i) ds\\
&\stackrel{\xi\in(t_i,s)}{\le} \frac{(T_2-T_1)^2}{2k^2}+\int_{t_i}^{t_{i+1}}\frac{\sinh(2K\xi)-2K\xi}{2\sinh^2(K\xi)}(s-t_i)ds\\
&\le\frac{(T_2-T_1)^2}{k^2},
\end{align*}
as done in the above, we can get the desired result.

\item{(3).} Denote $t_{k+1}=T_1+\frac{1}{k}(T_2-T_1)$. Since  \begin{align*}
\int_{t_i}^{t_{t+1}}ds\int_{t_i}^{s}e^{\frac{2Ku}{1+\beta}}du\le e^{\frac{2Kt_{i+1}}{1+\beta}}\int_{t_i}^{t_{t+1}}(s-t_i)ds=\frac{(t_{i+1}-t_i)^2}{2}e^{\frac{2Kt_{i+1}}{1+\beta}},
\end{align*}
it follows
\begin{align*}
\sum_{i=0}^{k-1}\int_{t_i}^{t_{t+1}}ds\int_{t_i}^{s}e^{\frac{2Ku}{1+\beta}}du&\le \frac{(T_2-T_1)^2}{2k}\left(1+\frac{1}{k}\sum_{i=0}^{k-1}\left(e^{\frac{2Kt_{i+1}}{1+\beta}}-1\right)\right)\\
&\le \frac{(T_2-T_1)^2}{2k}\left(1+\int_{t_1}^{t_{k+1}}e^{\frac{2Kx}{1+\beta}}-1dx\right)\\
&\le \frac{(T_2-T_1)^2}{2k}\left(1+\frac{1+\beta}{2K}\left(e^{\frac{2Kt_{k+1}}{1+\beta}}-e^{\frac{2Kt_{1}}{1+\beta}}\right)-(T_2-T_1)\right)\\
&=\frac{(T_2-T_1)^2}{2k}\left(1+\frac{1+\beta}{2K}e^{\frac{T_2-T_1}{k}}\left(e^{\frac{2KT_2}{1+\beta}}-e^{\frac{2KT_1}{1+\beta}}\right)-(T_2-T_1)\right),
\end{align*}
the desired follows easily.\end{proof}

\end{document}